\documentclass{amsart}
\usepackage{verbatim}

\numberwithin{equation}{section}

\newcommand{\RR}{{\mathbb R}}

\newcommand{\R}{{\mathbb R}}
\newcommand\Op{\operatorname{Op}}

\usepackage{epsfig}
\usepackage{amscd}
\usepackage{amsmath, amssymb, amsthm}
\usepackage{graphics}
\usepackage{color}
\usepackage{dsfont}
\newtheorem*{main-theorem}{Main Theorem}
\newtheorem{proposition}{Proposition}[section]
\newtheorem{theorem}{Theorem}
\newtheorem*{old-thm}{Theorem}
\newtheorem{lemma}[proposition]{Lemma}
\newtheorem{corollary}[proposition]{Corollary}

\theoremstyle{definition}
\newtheorem{definition}[proposition]{Definition}
\newtheorem*{remark}{Remark}
\numberwithin{equation}{section}

\def\11{\mathds{1}}

\def\reals{{\mathbb R}}

\def\supp{\mathrm{supp}\,}

\def\O{{\mathcal O}}

\def\Op{\mathrm{Op}\,}

\def\phi{\varphi}

\def\dist{\text{dist}\,}

\def\be{\begin{eqnarray*}}
\def\ee{\end{eqnarray*}}
\def\ben{\begin{eqnarray}}
\def\een{\end{eqnarray}}

\def\lll{\left\langle}
\def\rrr{\right\rangle}

\def\L2R{L_{\text{Rest}}^2}

\def\tchi{\tilde{\chi}}

\def\dist{\mathrm{dist}}

\def\skip#1{}

\begin{document}

  
  \title[Exterior mass estimates  and Neumann data]{Exterior mass estimates  and  $L^2$ restriction bounds
  for Neumann data along hypersurfaces}

\author{Hans Christianson}
\address{Department of Mathematics, UNC-Chapel Hill \\ CB\#3250
  Phillips Hall \\ Chapel Hill, NC 27599}
  \email{hans@math.unc.edu}
 \author{Andrew Hassell}
 \address{Mathematical Sciences Institute, Australian National University \\ Canberra 0200 AUSTRALIA}
 \email{Andrew.Hassell@anu.edu.au}
\author{John A. Toth}
\address{Department of Mathematics, McGill University\\ Montreal, CANADA}
\email{jtoth@math.mcgill.ca}

\begin{abstract}
We study the problem of estimating the $L^2$ norm of  Laplace eigenfunctions on a compact Riemannian manifold $M$ when restricted to a hypersurface $H$. 
We  prove  mass estimates for the restrictions of eigenfunctions $\phi_h$, $(h^2 \Delta - 1)\phi_h = 0$,  to $H$ in the region exterior to the coball bundle of $H$,   on $h^{\delta}$-scales ($0\leq \delta < 2/3$).  We use
this estimate to obtain an $O(1)$ $L^2$-restriction bound for the
Neumann data along $H.$  The estimate also applies to eigenfunctions
of semiclassical Schr\"odinger operators.
\end{abstract}
\maketitle


\section{Introduction}
\label{introduction}

 We consider here the
eigenvalue problem on a compact Riemannian manifold $(M,g)$ with or without boundary,
with either Dirichlet or Neumann boundary conditions if $\partial M
\neq \emptyset$.  That is, we consider 
$$\left\{ \begin{array}{l}  -\Delta_g \phi_j = \lambda_j^2 \phi_j,
    \;\;\; \text{on } M, \\ \langle \phi_j, \phi_k \rangle = \delta_{jk} \\
\\
B \phi_j = 0 \;\; \mbox{on}\;\; \partial M. \end{array} \right.$$
Here, $\Delta_g$ is the negative Laplacian associated with the metric $g$, $\langle f, g \rangle = \int_M f \bar{g} dV$ is the $L^2(M)$
inner product with respect to the induced Riemannian volume form $dV$,
and where $B$ is the boundary operator, 
either $B \phi = \phi|_{\partial M}$ in the Dirichlet case or $B
\phi = \partial_{\nu} \phi|_{\partial M}$ in the Neumann case.

 We introduce a hypersurface  $H \subset M$, which we assume to  be
 orientable,
embedded, and  separating in the sense that
$$M \backslash H = M_+ \cup M_-
$$
where $M_{\pm}$ are domains with boundary in $M$.
 This is not a  restrictive assumption since our argument is local, and every hypersurface is locally separating. 
 
 \skip{This note is concerned with small-scale (2-microlocal) control estimates for the eigenfunction Cauchy data along $H,$ where  
$$CD(\lambda):= (\phi_{\lambda}|_H, \lambda^{-1} \partial_{\nu} \phi_{\lambda}|_H).$$
Specifically, we present  two independent but closely-related results here.}


Our main result deals with $L^2$-restriction bounds for the normalized Neumann data $ \lambda^{-1} \partial_{\nu} \phi_\lambda |_H.$ 



\begin{theorem} \label{main}
Suppose $H \subset M$ is a smooth, 
embedded orientable separating hypersurface   and assume that $H \cap
\partial M = \emptyset$ if $\partial M \neq \emptyset$. Let $\{ \phi_{\lambda_j} \}_{j=1}^{\infty}$ denote the $L^2$-orthonormalized Laplace
eigenfunctions on $M.$ Then,
$$ \| \lambda_j^{-1} \partial_{\nu} \phi_{\lambda_j} \|_{L^2(H)} = {\O}(1).$$
\end{theorem}

Theorem \ref{main} generalizes a classical result for boundary traces of Dirichlet eigenfunctions  to {\em arbitrary} interior hypersurfaces (see e.g. Hassell and
Tao \cite{HT}). We note 
that the universal $L^2$-restriction upper bound in Theorem
\ref{main} for the normalized Neumann data $\lambda^{-1} \partial_{\nu} \phi_{\lambda} 
|_{H}$ is sharp (see Section \ref{S:sharp}) and is substantially better than for the corresponding Dirichlet
data $ \phi_{\lambda} |_{H}$ which, by \cite{BGT}, is only ${\mathcal
  O}(\lambda^{\frac{1}{4}}).$ The latter estimate is also sharp.

We actually prove a more general result here that applies to eigenfunctions of general semiclassical Schr\"odinger operators with Theorem \ref{main} a special case. 

\begin{theorem}\label{main2}
Let $P(h)$ denote the Schr\"odinger operator $P(h) =-h^2 \Delta_g + V(x)$ with $V \in C^{\infty}(M;\R)$. Let $E$ be a regular value of the symbol $p(x,\xi) = |\xi|_g^2 + V(x),$
 and let $\phi_{h} \in C^{\infty}(M,g)$ be a sequence of $L^2$-normalized eigenfunctions of  $P(h)$ with eigenvalues $E(h)$ in the interval  $[E-C h, E+ Ch]$.  Then, provided $H$ is an oriented separating hypersurface with $V(x) < E$ for $x \in H$, we have  
$$ \| h \partial_{\nu} \phi_h \|_{L^2(H)} = {\O}(1)$$
as $h \to 0.$
\end{theorem}

For simplicity, we give the proof of Theorem \ref{main} first and then outline the fairly minor changes for Schr\"odinger eigenfunctions in section \ref{schroedinger}.

We should point out that although not explicitly stated in Tataru's paper, the  result for Laplace eigenfunctions in Theorem \ref{main} actually follows from \cite{Ta} Theorem 2 if one consider  wave functions of the form $u(x,t) = e^{i\lambda t} \phi_{\lambda}(x).$ Then, the regularity bounds for $u(x,t)$ in \cite{Ta} readily yield the semiclassical estimates for the eigenfunctions $\phi_{\lambda}(x)$ in our Theorem \ref{main}.  However, our proof of Theorem \ref{main} here is quite different  and our general result in Theorem \ref{main2} is apparently new.
We remark also that we have been informed that Melissa Tacy has also
obtained an independent 
proof of Theorem \ref{main} using different methods in a current work in progress. 

Theorems \ref{main} and \ref{main2} follow  from a Rellich commutator argument together with 
a `small-scale' estimate for the mass of an eigenfunction in the region exterior to the coball bundle on $H$ (Proposition~\ref{prop:mass-decay}). Here the significance of the coball bundle is that it is the projection of the characteristic variety of the semiclassical Laplacian $-h^2 \Delta-1$ to $T^* H$, hence the region of phase phase where the eigenfunctions are expected to concentrate. By `small-scale' we mean we localize outside a neighbourhood of the coball bundle of size $h^\delta$, where $\delta$ is allowed to be larger than $1/2$; in fact, we find that one can let $\delta$ be as large as $2/3 - \epsilon$: we show that outside a neighbourhood of this size, the mass is $O(h^\infty)$. Also, we observe that the exponent $2/3$ is optimal; in fact, the mass outside an $h^{2/3}$-sized neighbourhood of the coball bundle can be as large as $h^{1/6}$ as we show with a simple example.

In the following we let the semiclassical parameter $h
\in \{\lambda_j^{-1} \}_{j=1}^\infty$ and rescale the Laplacian to the semiclassical operator $P(h) = -h^2 \Delta_g-1.$ We abuse notation somewhat and write $\phi_{h} = \phi_{\lambda}$ for the eigenfunction with eigenvalue $\lambda^2 = h^{-2}.$
 To state the relevant commutator estimate, we introduce various $h$-pseudodifferential cutoffs suppressing for the moment some of the technical details. Let $\chi \in C^{\infty}_{0}(\R;[0,1])$ with $\chi(u) = 1 $ for $|u| \leq 1/2$ and $\chi(u) = 0$ for $|u| >1,$   $\chi_{-} \in C^{\infty}(\R) $ with $\chi_{-}(u) = 1 $ when $u < -1$ and $\chi_{+} \in C^{\infty}(\R)$ with $\chi_{+}(u) =1$ when $u > 1.$ In addition we require that 
$$ \chi_{-}(u) + \chi(u) + \chi_{+}(u) =1; \,\,u \in \R.$$ 
  Let $R(x',\xi') = \sigma (-h^2\Delta_H)(x',\xi')$ be the principal symbol of the induced hypersurface Laplacian $-h^2\Delta_H: C^{\infty}(H) \rightarrow C^{\infty}(H)$ and consider the decomposition of  $T^*H$ into 3 pieces given by the   radial cutoffs $\chi_{in}, \chi_{tan},\chi_{out} \in C^{\infty}(\R;[0,1])$ 
 with $\chi_{in}(x',\xi') = \chi_{-}( R(x',\xi') -1 ),  \chi_{tan}(x',\xi') = \chi( R(x',\xi') -1)$ and $\chi_{out}(x',\xi') = \chi_{+}(R(x',\xi') -1)$.  Clearly, supp $\chi_{in}  \subset B^*H$, supp $\chi_{out} \in T^*H-\overline{B^*H}$ and in addition,
\begin{equation} \label{cutoff decomp}
\chi_{in}(x',\xi') + \chi_{tan}(x',\xi') + \chi_{out}(x',\xi') = 1; \,\,\, (x',\xi') \in T^*H. \end{equation}
For any $\delta \in [0,1)$ we also define the rescaled cutoff functions by $(\chi_{in})_{h,\delta} (x,\xi) = \chi_{-} ( h^{-\delta}( R(x',\xi') -1) ),  (\chi_{tan})_{h,\delta} (x,\xi) = \chi (h^{-\delta}( R(x',\xi') - 1) ),$ and  $(\chi_{out})_{h,\delta} (x,\xi) = \chi_{+} (h^{-\delta}( R(x',\xi') - 1) ).$

We denote the corresponding $h$-Weyl pseudodifferential partition of
unity by $(\chi_{tan})_{h,\delta}^w = Op_h^w ( (\chi_{tan})_{h,\delta}
)$ and similarly for $(\chi_{out})_{h,\delta}^w$ and
$(\chi_{in})_{h,\delta}^{w}$ (see Figure \ref{fig:cutoffs}). In the following, we denote the canonical restriction map by $\gamma_H: C^{\infty}(M) \rightarrow C^{\infty}(H)$ and the corresponding eigenfunction restriction by $\phi_h^{H}:= \gamma_H \phi_h \in C^{\infty}(H).$

\begin{figure}
\hfill
\centerline{\input{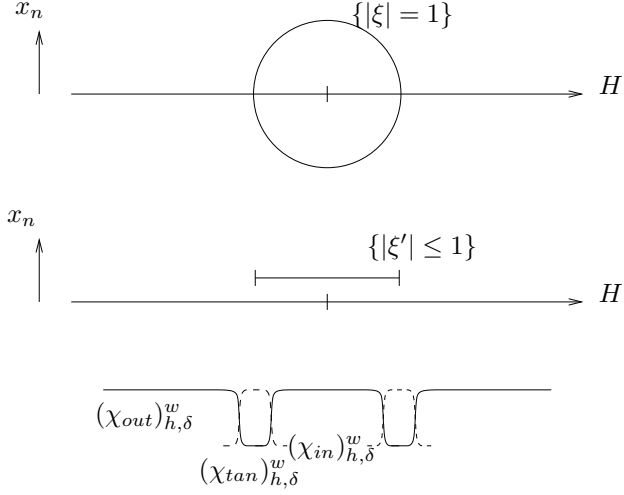}}

\caption{\label{fig:cutoffs} The hypersurface $H$ in Fermi normal
  coordinates with the restricted cosphere bundle.  In the second
  picture we have projected onto the coball bundle of $H$ and sketched
  the microlocal partition of unity.  The cutoff
  $(\chi_{in})_{h,\delta}^{w}$ microlocalizes to the interior of the
  coball, on scale $h^\delta$ from the edge,
  $(\chi_{out})_{h,\delta}^{w}$ microlocalizes to the exterior on the
  same scale, and $(\chi_{tan})_{h,\delta}^{w}$ (in dashed)
  microlocalizes to the edge of this set, where $| \xi' | \sim 1$ on
  scale $h^\delta$.  This is the ``glancing'' set, since these
  directions have little or no normal direction $\xi_n$ in the lifted 
  cosphere bundle.}
\hfill
\end{figure}

Let us explain how our argument begins. We use a Rellich identity, involving the commutator of $-h^2 \Delta - 1$ with the operator 
$\chi(x_n) hD_n$. Integrating over $M_-$, we  have, using Green's formula, 
\begin{align*}
\frac{i}{h}  \int_{M_-} [-h^2 \Delta -1, &\chi(x_n) hD_n] \phi_h \overline{\phi_h}
dx \\
 = \int_H ((hD_n)^2  \phi_h)|_H \overline{\phi_h}|_H d \sigma_H 
& + \int_H (h D_n \phi_h)|_H \overline{hD_n \phi_h}|_H d \sigma_H .
\end{align*}
The LHS is $O(1)$ as it involves  the expectation value of $\phi_h$ with a second order differential operator. The second term of the RHS is exactly the quantity we seek to bound. Thus we need to understand the first term on the RHS. Since $\phi_h$ is an eigenfunction, $(h D_n)^2 \phi_h$ is equal to $(1 + h^2 \Delta_H) \phi_h$ up to an error term $h^2 L \phi_h^H$ where $L$ is a first order differential operator. We can ignore the $h^2 L$ term using the Burq et al $O(h^{-1/4})$ bound for $\phi_h^H$. The main observation that makes the proof work is that $\phi_h^H$ is semiclassically localized inside the coball bundle of $H$, and on this set, $1 + h^2 \Delta_H$ has nonnegative symbol. Hence the first term on the LHS is morally a positive term, that is, of the same sign as the second term. This leads immediately to the $O(1)$ estimate we seek. Thus the main work in the proof is justifying that the $(1 + h^2 \Delta_H) \phi_h^H$ term is indeed positive up to an $O(1)$ error. To do this, we break up this term using the partition of the identity operator given by $(\chi_{in})_{h,\delta}^w ,  (\chi_{tan})_{h,\delta}^w$ and $(\chi_{out})_{h,\delta}^w$ defined above, for $\delta > 1/2$, and analyze each separately.

\section{Second microlocalization at a hypersurface}

In the proof of Theorem \ref{main} we will use  $h$-pseudodifferential operators second microlocalized (2-microlocalized) along the hypersurface $H.$  For the convenience of the reader, we collect and briefly review here the requisite semiclassical
analysis including the  various 2-microlocal symbol classes and the corresponding pseudodifferential operator calculus. The material here is a special case of a more general two-parameter calculus developed in \cite{SjZw1,SjZw2}.  Since our interest lies in establishing $L^2$-restriction bounds for eigenfunctions along a hypersurface $H \subset M,$ we need only consider ambient symbols supported in an $\epsilon >0$ neighbourhood of $H$, where $\epsilon >0$ is arbitrary small. Thus, we introduce Fermi coordinates $x = (x',x_n)$ near $H$ with $ x= \exp_{x'}(x_n \nu_{x'})$ with $\nu_{x'}$ an exterior unit normal to $H$. Since by assumption $H$ is orientable, this is well-defined. In Fermi coordinates,  $H = \{ x_n = 0 \}$  and it is convenient to define our symbols in terms of these coordinates. We do so without further comment.

\subsection{ Homogeneous and semiclassical symbol classes} We collect
for  future reference a brief review of  the standard symbol classes
and corresponding pseudodifferential operators used later on (see also
\cite[Section 4.4]{Zw-book}). The more subtle 2 microlocal semiclassical analysis is treated in \ref{2microlocal}.

The standard homogeneous symbol spaces that are relevant here are
\begin{align} 
S^{m}_{\rho,\delta} & (T^*M) \notag \\
& = \{ a(x,\xi) \in C^{\infty}(T^*M-0);
|\partial_{x}^{\alpha} \partial_{\xi}^{\beta} a(x,\xi)| = {\mathcal
  O}_{\alpha,\beta}( \langle \xi \rangle^{m-\rho|\alpha| + \delta
  |\beta|}), \, \rho > \delta \} \label{symbols1}
\end{align}
As for the semiclassical symbols,   the relevant symbol classes for our purposes are
\begin{align} \label{symbols2} 
S^{m,k}_{cl} & (T^*M \times (0,h_0])  
 = \{ a(x,\xi;h) \in C^{\infty}(T^*M \times (0,h_0]); \\
& a(x,\xi;h)
\sim \sum_{j=0}^{\infty} a_{m-j}(x,\xi) h^{-m+j} , \,\, a_{m-j} \in
S^{k}_{0,0}\}, \notag \\  
S^{m}_{\delta} & (T^*M \times (0,h_0])\\
& = \{ a(x,\xi;h) \in C^{\infty}(T^*M \times (0,h_0]; |\partial_{x}^{\alpha} \partial_{\xi}^{\beta} a| = {\mathcal O}_{\alpha,\beta}( h^{-m}   h^{- \delta(|\alpha| + |\beta|)} \langle \xi \rangle^{-\infty} ) \}, \notag
\end{align}
with $\delta \in [0,1).$
Since both the eigenfunctions $\phi_h$ and their restrictions $u_h =
\phi_h|H$ have compact $h$-wavefront sets (see for example
\cite[Section 8.4]{Zw-book} and
Section \ref{S:loc} below) we are interested here in only the case
where the $\xi$ variables are in a compact set. Consequently, the semiclassical symbol classes $S^{m}_{\delta}$ are most relevant. In the special case where $\delta =0,$ 
$$ S^{m}_{0}(T^*M \times (0,h_0]) = \{ a \in C^{\infty}(T^ *M \times (0,h_0]); | \partial_x^{\alpha} \partial_{\xi}^{\beta} a| = {\mathcal O}_{\alpha,\beta} (h^{-m} \langle \xi \rangle^{-\infty}) \}.$$
When the context is clear, we sometimes just write $S^{m}_{\delta}$
instead of $S^{m}_{\delta}(T^*M \times (0,h_0]).$ The case where
$\delta = 0$ is sometimes denoted by $S^{m}(1)$ in the literature.

The corresponding $h$-Weyl pseudodifferential operators have Schwartz kernels that are sums of the local integrals of the form
\begin{equation} \label{psdo}
Op_{h}^{w}(a)(x,y)  = (2\pi h)^{-n} \int_{\R^n} e^{i \langle x-y,\xi \rangle/h } a( \frac{x+y}{2},\xi;h) d\xi. \end{equation}
We will use $Op_{h}^{w}(a),$ $a_{h}^w$ and $a^w(x,h D_x)$ interchangeably to denote $h$-Weyl quantizations of $a(x,\xi;h)$  since each has its advantages. It is standard that for $a \in S^{m_1,k_1}_{cl}, b \in S^{m_2,k_2}_{cl},$
$$ a^w(x,h D_x) \circ b^w(x,h D_x) = e^{i h \sigma(D_x,D_\xi,D_y,D_\eta)/2} a(x,\xi) b(y,\eta) |_{y=x,\eta=\xi}$$
$$= c^w(x,h D_x) \in Op_{h}^w(S^{m_1+m_2,k_1 + k_2}_{cl}(T^*M) )$$ with
$ c(x,\xi;h) = a(x,\xi;h) \# b(x,\xi;h)$ and $\sigma(x,\xi,y,\eta) = y \xi - x \eta.$ 
Similarily, for  $a \in S^{m_1}_{\delta}, b \in S^{m_2}_{\delta}$ with $\delta \in [0,1/2),$ 
$$ a^w(x,h D_x) \circ b^w(x,h D_x) = c^w(x,h D_x) \in Op_{h}^w(S^{m_1+m_2}_{\delta}(T^*M) )$$ with $ c(x,\xi;h) = a(x,\xi;h) \# b(x,\xi;h).$  

Since eigenfunctions (and their restrictions) have compact
$h$-wavefront, it is the algebra $Op_{h}(S^*_{\delta})$ that is most
relevant here.  We point out that for $a^w(x,h D_x) \in
Op_{h}(S^0_{\delta})$, $0 \leq \delta < 1/2$, with $a(x,\xi;h) \geq
0,$ one also has the sharp G{\aa}rding inequality $a^{w}(x,h D_{x}) \geq -
C h^{1 - 2 \delta}$ (in the $L^2$ sense) and indeed the sharper Fefferman-Phong inequality 
\begin{equation} \label{garding}
a^{w}(x,h D_{x}) \geq - C h^{2- 4 \delta} \end{equation}
also holds \cite[Section 4.7]{Zw-book}.

\subsection{Semiclassical second-microlocal pseudodifferential cutoffs: microlocal decompostion}
\subsubsection{Fermi normal coordinates near $H$}  \label{fermi} Frow now on, we let $x = (x',x_n) $ be Fermi normal coordinates in a small
tubular neighbourhood $H(\epsilon)$ of $H $ defined near a point
$x_0 \in H$. In these coordinates we can locally write
$$H(\epsilon) := \{ (x',x_n) \in U \times {\mathbb R}, \, | x_{n} | < \epsilon \}.$$
Here $U \subset {\mathbb R}^{n-1}$ is a coordinate chart
containing $x_0 \in H$ and $\epsilon >0$ is arbitrarily small but
for the moment, fixed. We let $\chi \in C^{\infty}_{0}({\mathbb
R})$ be a cutoff with $\chi(x) = 0$ for $|x| \geq 1$ and $\chi(x)
= 1 $ for $|x| \leq 1/2.$
  Moreover, in terms of the normal coordinates,
   $$-h^2\Delta_g = \frac{1}{g(x)} hD_{x_n} g(x) hD_{x_n}   + R(x_n,x',hD_{x'})$$
where $R$ is a second-order $h$-differential operator along $H$
with coefficients that depend on $x_{n}$, and $R(0,x', hD_{x'})$
is the
induced tangential semiclassical Laplacian, $-h^2\Delta_H,$ on $H$.  Consequently, at the level of symbols,
$$ \sigma(-h^2 \Delta_{g})(x,\xi)  = |\xi|^{2}_{g} = \xi_n^2 + R(x',x_n,\xi'),$$
where,
$$ \sigma(-h^2 \Delta_H)(x',\xi') = R(x',0,\xi').$$

Let $\gamma_H: C^{0}(M)\rightarrow C^{0}(H)$ be the restriction operator $\gamma_H(f) = f|_{H}.$ The adjoint $\gamma_H^*:C^{0}(H) \rightarrow C^0(M)$ is then given by
$$ \gamma_H^* (g) = g \cdot \delta_H.$$
When there is no risk for confusion we write $\phi_{h}^{H} = \gamma_H \phi_{h}$ and similarily $\phi_{h}^{H,\nu} = -ih\gamma_H \partial_{x_n} \phi_{h}.$

We now describe the relevant 2-microlocal $h$ pseudodifferential operators that are  second microlocalized along $ \Sigma_H:=S^*H \subset T^*H.$ 

\subsubsection{ Semiclassical  pseudodifferential operators second microlocalized along $\Sigma_H$} \label{2microlocal}

We  introduce here the relevant 2-microlocal algebra of
$h$-pseudodifferential operators localized on  small scales $\sim
h^{\delta}$ where $  \delta  \in (1/2,1)$ that will be used in the proof of
Theorem \ref{main}. When $ \delta  \in
(0,1/2)$ the pseudodifferential calculus is well-known \cite[Chapter 4]{Zw-book} but for $\delta > 1/2$ the construction is more subtle. The relevant calculus has been developed in very general framework by Sj\"ostrand and Zworski \cite{SjZw1,SjZw2} to which we refer the reader for further details.  Since a rather simple special case of their calculus will suffice for our purposes,  we will attempt to  keep the argument fairly self-contained.  

\begin{definition}  \label{2micropsdos} Let $H \subset M$ be a
  hypersurface. We say that a semiclassical symbol $b$ is
  2-microlocalized along $\Sigma_H$ and write $b \in
  S^{m}_{\Sigma_H,\delta}(T^*H \times (0,h_0])$  provided there exists
  $ \chi  \in C^{\infty}_{0}(\R), a_1(x',\xi';h) \in
  S^{0}( T^*H \times (0,h_0] )$  such that
$$b(x',\xi';h) = h^{-m} a_1(x',\xi';h) \cdot \chi \left( \frac{R(x',x_n=0,\xi')-1}{h^{\delta}} \right), \,\,\, 0 \leq \delta < 1$$
for all $(x',\xi') \in T^*H.$
\end{definition}

We will need the following  proposition (see also \cite{SjZw1,SjZw2}).
\begin{proposition} \label{algebra} Given $a^w(x,h D_x) \in Op_h (S^{m_1}_{\Sigma_H,\delta})$ and $b^w(x,h D_x) \in Op_h (S^{m_2}_{\Sigma_H,\delta})$ it follows that
 $$ a^{w}(x,h D_x) \circ b^w(x,h D_x)  = c^w(x,hD_x) \in Op_h (S^{m_1 + m_2}_{\Sigma_H,\delta})$$ with
$$ c(x,\xi;h) = a(x,\xi,h) \# b(x,\xi,h).$$
\end{proposition}
\begin{proof}  Since
$$ d_{\xi'} R(x',0,\xi')  \geq C \langle \xi' \rangle, \,C>0$$
near $\{ R(x', 0, \xi') = 1 \}$, 
as in \cite{SjZw1,SjZw2}, the proof hinges on the following
real-principal type quantum normal form construction given in the
following Lemma.  
\begin{lemma} \label{normal form} Let $(x_0,\xi_0) \in \Sigma_H$ and let
  $U \subset T^*H$ be a sufficiently small open neighbourhood of $(x_0,\xi_0).$ For $U$ small, there exists $V \subset T^*\R^{n-1}$ open together with a canonical transformation
$$\kappa_F: (U, (x_0',\xi_0')) \longrightarrow (V; (0,0)); \,\,\, \kappa_F(x',\xi') = (y',\eta') $$ and corresponding $h$-Fourier integral operators $F(h):C^{\infty}_{0}(U) \rightarrow C^{\infty}_{0}(V)$ such that
$$ (i) \,\, F(h)^{*} \circ  \chi^w_{h} \left( \frac{ R(x',0,\xi') - 1}{h^{\delta}} \right) \circ F(h) =_{U \times V} {\chi}^w_h \left( \frac{\eta_1'}{h^{\delta} }; h \right),$$
$$ (ii) \,\,\, F(h)^* \circ F(h)=_{U \times V} Id,$$
with ${\chi} \in C^{\infty}_{0}(\R^{n-1};(0,0) \times (0,h_0] )$.   Here, $A(h)=_{U \times V} B(h)$ denotes $h$-microlocal equivalence on $U \times V \subset T^*H \times T^*\R^{n-1}.$
\end{lemma}

Given Lemma \ref{normal form} one reduces the proof of the proposition
to operators in normal form. In the conormal variables $(x_n,\xi_n)$
the composition formula is standard since symbols are  in the standard
$S^{0}_{\delta/2}$-classes with $ \delta/2 < 1/2$.  Since symbols in $S^{m}_{\Sigma_H,\delta}$ are separable, it suffices to assume that
$a(x,\xi;h) = \chi_1 ( h^{-\delta} ( R(x',0,\xi') - 1) )$ and
$b(x,\xi;h) = \chi_2( h^{-\delta} ( R(x',0,\xi') - 1) )$ with $\chi_j
\in C^{\infty}_{0}(\R); j =1,2.$  Let $\psi_{U_j}\times \psi_{V_j}
\in Op_h( S^{0}_{0}) ; j=1,...,N_0$ be an $h$-microlocal partition
of unity subordinate to a covering of the supports of $a$ and $b$ in
$T^*H$ by open sets $U_j; j=1,...,N_0$ such that on each $U_j$
Lemma \ref{normal form} holds with h-FIO $F(h)$ with $WF_h'(F(h))
\subset U_j \times V_j.$  
From Lemma \ref{normal form},
$$ a^w_h \circ b^w_h = \sum_{j=1}^{N_0} \psi_{U_j}^w F(h) \circ [ F(h)^* a^w_h F(h)] \circ [F(h)^* b_h^w F(h)] \circ F(h)^* \psi_{V_j}^w + R(h),$$
where $R(h) \in Op_h(S^{-\infty}_{0}).$ For the inner model operators
one simply rescales the fiber variables $(y',\eta'/h^{\delta}) \mapsto
(y'/h^{\delta/2}, \eta'/h^{\delta/2})$ and computes the composition in
the model normal coordinates. The result is that in each chart 
\begin{align} \label{model1} 
 [ F(h)^* & a^w_h F(h)]  \circ [F(h)^* b_h^w F(h)]  \\
& =     \left( \chi_1 (h^{-\delta} \eta_1')  \# \chi_2 (h^{-\delta} \eta_1')  \right)^{w} + {\mathcal O}(h^{\infty})_{L^2 \rightarrow L^2} \notag
   \\ &  =  (2\pi h )^{-(n-1)} \int_{\R^{n-1}} e^{i \langle
     x'-y',\eta' \rangle/ h}
   \psi_{V_j}(\frac{x'+y'}{2h^{\delta/2}},h^{\delta/2}\eta')  (\chi_1
   \# \chi_2) ( h^{-\delta/2} \eta_1')   \, d\eta' \\
& \quad \quad + {\mathcal O}(h^{\infty})_{L^2 \rightarrow L^2}.  \notag\end{align}
   The $\#$ product expansion is computed as usual in the
   $h^{-\delta/2}$ calculus and then rescaled, which is particularly
   simple for our special choice of operators:
   $$ (\chi_1 \# \chi_2) ( h^{-\delta} \eta_1')  = \sum_{j=0}^{\infty} h^{j} ( D_x D_{\eta} - D_y D_{\xi})^{j} [ \chi_1(h^{-\delta} \eta_1')  \cdot  \chi_2 ( h^{-\delta} \xi_1') ] |_{\xi' = \eta'} $$
$$ = \chi_1(h^{-\delta} \eta_1')  \cdot  \chi_2 ( h^{-\delta} \eta_1') + R(y',\eta';h),$$
where, $R(y',\eta';h) \in S^{-\infty}_{0}(T^*\R^{n-1}).$

\end{proof}

\begin{remark} We note that $L^2$ boundedness for $a^w_h \in Op_h (S^0_{\Sigma_H,\delta})$ with $\delta \in [0,1)$ is clear by passing to normal form. Since $F(h)^* F(h) \cong_{U_j \times V_j} Id,$
\begin{align} 
\label{bdd}
\| a^w(x,hD_x) \|_{L^2 \rightarrow L^2} & \leq  \sum_{j=1}^{N_0} \| \chi_h^w(h^{-\delta} \eta'_{1}) \psi_{V_j,h}^w(y',\eta') \|_{L^2 \rightarrow L^2}  \\ 
& = \sum_{j=1}^{N_0} \| \chi^w(h^{1-\delta} \eta'_{1})
\psi_{V_j}^w(y',h \eta') \|_{L^2 \rightarrow L^2} \notag \\
& \leq \sum_{| \alpha |\leq 2n +1}  \|\partial^\alpha  \chi( h^{1-\delta} \eta'_{1}) \psi_{V_j}(y',h \eta')
\|_{L^{\infty}} \notag \\
& \leq \| \chi( h^{1-\delta} \eta'_{1}) \psi_{V_j}(y',h \eta')
\|_{L^{\infty}} + {\mathcal O}(h^{1-\delta}). \notag
\end{align}
This follows by  rescaling $\eta' \mapsto h \eta'$ in the
fiber variables and then applying the Calderon-Vaillancourt theorem,
together with the trivial estimates on the derivatives of the
functions $\chi$ and $\psi_{V_j}$.

 \end{remark}

 The proof of Lemma \ref{normal form}, in particular the  local
 rescaling 
\[
(y', \eta' /h^\delta) \mapsto ( y' / h^{\delta/2}, \eta'/
 h^{\delta/2} ), 
\]
can also be used to prove a version of the
 G{\aa}rding inequalities for operators in $\Op_h (S^m_{\Sigma_H,
   \delta} )$.
\begin{lemma}
\label{L:garding}
Suppose $a \in S^0_{\Sigma_H, \delta}$ is real valued and $a \geq 0$.
Then 
\[
\lll a^w u, u \rrr \geq -C h^{1-\delta} \| u \|^2.
\]
In particular, 
\[
\lll a^w|_H \phi_h^H, \phi_h^H \rrr_{L^2(H)} \geq -C h^{1-\delta} \|
\phi_h^H \|^2_{L^2(H)}.
\]

\end{lemma}

\section{Eigenfunction energy localization: estimating the exterior mass} 
\label{S:loc}

Let
$M, H$ and $\phi_{\lambda_j}; j=1,2,...$ be as in the Introduction. 
A key part of the proof of Theorem~\ref{main} involves estimating the mass of restricted eigenfunctions $\phi_h |_H$ in regions exterior to the coball bundle of $H$ and on 2-microlocal scales $\delta >1/2$ (see section \ref{mainthm}).   In this section, for the benefit of the reader, we first review some known results on mass concentration for the ambient eigenfunctions on M. Then, in Proposition \ref{prop:mass-decay}, we give the analogous mass estimates for the restricted eigenfunctions $\phi_h |_H.$ The latter results appear to be new.


\subsection{Exterior mass estimates on $M.$} \label{1 scales}

  Let $h \in \{ \lambda_j^{-1} \}; j=1,2,...$. Given $0 < \epsilon_0
  <1$ an arbitrary small number, let $\chi(x,\xi) \in
  C^{\infty}_{0}(T^*M)$ be equal to one on the annulus $A(\epsilon_0)= \{ (x,\xi);  (1-\epsilon_0/10)  < |\xi|_g < (1 + \epsilon_0/10)$ and with supp $\chi \subset A(2\epsilon_0).$ Let $\tilde{\chi} \in C^{\infty}_{0}$ be another cutoff equal to one on $A(2\epsilon_0)$ and with supp $\tilde{\chi} \subset A(4\epsilon_0).$   Consider  the eigenfunction equation
$$ (-h^2 \Delta_g - 1) \phi_{h} = 0.$$
Then, $P(h):= -h^2 \Delta_g - 1$ is $h$ elliptic for $(x,\xi) \in T^*M - A(\epsilon)$. So, one can construct  an $h$-microlocal parametrix with $Q(h) \in Op_{h}(S^{0}_{0,0})$ so that
$$(1- \tilde{\chi}(h)) Q(h)  P(h) (1-\chi(h)) \phi_{h} =
(1-\tilde{\chi}(h)) \phi_{h} + {\mathcal O}(h^{\infty}).$$ Since $P(h)
\phi_{h} = 0$  and $\sigma( [ P(h), (1-\chi(h))] (x,\xi) = 0$ for
$(x,\xi) \in $ supp$ (1-\tilde{\chi}(h)),$  one gets the well-known concentration estimate
\begin{equation} \label{con1} 
 \|  (1-\tilde{\chi}(h)) \phi_{h} \|_{L^2} = {\mathcal O}(h^{\infty})
 .
\end{equation}
 A similar argument with the derivatives $\partial_x^{\alpha} \phi_h(x)$ combined with Sobolev embedding implies
 \begin{equation} \label{con1.1}
  \|  (1-\tilde{\chi}(h)) \phi_{h} \|_{C^k} = {\mathcal
    O}(h^{\infty}),\end{equation}
and as a consequence
\begin{equation*}
WF_{h}(\phi_{h}) \subset S^*M. 
 \end{equation*}
 When $\delta \in (0,1/2),$ the parametrix construction above extends to cutoffs $\chi_{\delta}(h) \in Op_{h}(S^{0}_{\delta})$ without change and consequently, so do the mass estimate (\ref{con1}) and (\ref{con1.1}). When $\delta \in [1/2,1),$ the same is true provided one uses the $h$-pseudodifferential calculus 2-microlocalized along the hypersurface $\Sigma = S^*M \subset T^*M.$ One way to do this is as follows: we choose $\psi \in \mathcal{S}(\RR)$ with $\psi(0) = 1$ and the Fourier transform of $\psi$ compactly supported, and  write the 
operator $\chi_{\delta}(h):= \psi((-h^2 \Delta - 1)h^{-\delta})$ in terms of the Fourier Transform of $\psi$ and the semiclassical propagator of $-h^2 \Delta$. We can thus derive that this operator is in the 2-microlocal calculus $\Psi^0_{S^*M,\delta}(M)$ with exponent $\delta$. By construction, $\psi((-h^2 \Delta - 1)h^{-\delta}) \phi_h = \phi_h$. Now we apply an operator $A$ in $\Psi^0_{S^*M,\delta}(M)$ supported where $|\xi|_g \geq 1 + C h^\delta$ to $\psi((-h^2 \Delta - 1)h^{-\delta}) \phi_h$. The symbol calculus for $\Psi^0_{S^*M,\delta}(M)$ shows that for $C$ sufficiently large, the symbol of the composition is $O(h^\infty)$, and therefore $A \phi_h = O(h^\infty)$.  
To summarize, for eigenfunction masses on $M$, one has the exterior mass estimate 
\begin{equation}\label{massM}
\| ( 1 - \chi_{\delta}(h)) \phi_h \|_{L^2(M)} = O(h^\infty), \,\,\,\, 0 \leq \delta <1.
\end{equation}

\subsection{Exterior mass estimates on $H$} \label{extmass}

We will need to control exterior mass of eigenfunction restrictions to $H$ in the 2-microlocal setting where $\delta > 1/2.$ Unlike the case of the ambient manifold $M$ in (\ref{massM}), the allowable range of 2-microlocal scales (ie. the range of $\delta$'s) depends on the curvature of $H$ inside $M.$

In this section, we prove the following

\begin{proposition}\label{prop:mass-decay}
Let $(M,g)$, $H$ and $\phi_h^H$ be as above. 
Suppose that $0 \leq \delta < 2/3$, and let $(\chi_{out})_{h,\delta}^{w}$ be as defined as in the Introduction. Then 
\begin{equation}\label{massH}
(\chi_{out})_{h,\delta}^{w} \phi_h^H = O(h^\infty) \text{ in } L^2(H).
\end{equation}
Moreover, in the case where $H$ is totally geodesic, the exterior mass estimate (\ref{massH}) holds for all $0  \leq \delta < 1.$
\end{proposition}

\begin{proof}
In view of the discussion in Section~\ref{1 scales}, it suffices to prove that 
\begin{equation}
L_h (\chi_{out})_{h,\delta}^{w} \phi_h^H = O(h^\infty) \text{ in } L^2(H)
\end{equation}
for every pseudodifferential operator $L$ of semiclassical order zero and compact microsupport.

By definition of the Sj\"ostrand-Zworski calculus, $L_h(\chi_{out})_{h,\delta}^{w}(x',hD_{x'})$ is given in a small microlocal neighbourhood of a point $q \in T^* H$ by an expression of the form 
$$
(2\pi h)^{-(n-1)} \int e^{i(y-y') \cdot \eta/h} \chi_+\big( \frac{\eta_1}{h^\delta}\big) \, d\eta
$$
after conjugation by a semiclassical  FIO $T_h : L^2(H) \to L^2(\mathbb{R}^{n-1})$ associated to a canonical transformation $\rho : T^* H \to T^* \mathbb{R}^{n-1}$, defined in a neighbourhood of $WF_h(L_h)$,  that `straightens' $\Sigma = S^* H$ locally, in the sense that $\rho^*(\eta_1) = |\cdot|_{\tilde g}^2 - 1$ (where $\tilde g$ denotes the metric induced on the fibres of $T^*H$). (Recall that $\chi_+(t)$ is equal to $0$ for $t \leq 1/2$ and $1$ for $t \geq 1$.)

If $T_h$ is parametrized locally near $q$ by the phase function $\Phi(y, y', v)$ then we have 
\begin{equation}\begin{gathered}
L_h(\chi_{out})_{h,\delta}^{w}(x',hD_{x'}) \phi_h^{H} = \\
T_h^{-1}  \circ h^{\ast} \int    e^{i(y - y'')\cdot \eta/h} \chi_+\big( \frac{\eta_1}{h^\delta}\big) e^{i\Phi(y'', y', v)/h} a_1(y'', y', v, h) \phi_h^{H}(y') \, dy' \, dv  \, dy''\, d\eta,
\end{gathered}\end{equation}
for some exponent $\ast$ (depending on the number of components of $v$) which is not relevant, as we are about to show an $O(h^\infty)$ estimate. 

Next, we express $\phi_h^H$ locally in terms of its values on a surrounding annulus (up to $O(h^\infty)$ errors). To do this, we use Sogge's approximate projection operator  $\tilde \chi_\lambda$ \cite[Section 5.1]{Sogge-FICA}. Let us recall that $\tilde \chi_\lambda$ is defined to be the operator $\chi(\sqrt{\Delta} - \lambda)$, where $\chi$ is a Schwartz function with Fourier transform $\hat \chi$ having support in the interval $[\epsilon_0/2, \epsilon_0]$, where $\epsilon_0$ is chosen small enough (it suffices to take $\epsilon_0$ smaller than the injectivity radius of $M$). Then, as shown in \cite{Sogge-FICA}, the Schwartz kernel of $\chi_\lambda$ takes the form 
$$
\chi_\lambda(x, y) = \lambda^{(n-1)/2} a_2(y, x, \lambda) e^{-i\lambda \dist(x, y)} + R(x, y, \lambda),
$$
where $a_2$ is smooth with all derivatives bounded uniformly in $\lambda$, and is supported where $d(x,y) \in [(2C_0)^{-1} \epsilon_0, 2C_0 \epsilon_0]$ for some $C_0 > 1$. On the other hand, $R$ is smooth with all derivatives $O(\lambda^{-N})$ for every $N$. 

Let $h = \lambda^{-1}$. If we scale $\chi$ so that $\chi(0) = 1$, then we have
$$
\chi_{h^{-1}} \phi_h = \phi_h.
$$

We may assume without loss of generality that the projection of the microsupport of $L_h$ to $M$ is contained in a coball $B(p, r)$ of radius $r$, where $r + 2C_0 \epsilon_0$ is smaller than the injectivity radius. In that case, we can write, using a single coordinate patch, 
\begin{equation} \label{wkb}
\phi_h(x) = h^{-(n-1)/2} \int e^{-i\dist(x,y)/h} a_2(y, x, h) \phi_h(y) \, dy + O(h^\infty), \quad x \in B(p, r), 
\end{equation}
where $a_2$ is smooth with all derivatives bounded uniformly in $h$ and supported in $B(p, \iota)$ where $\iota$ is chosen smaller than the injectivity radius. It follows that we can write 
\begin{equation}\begin{gathered}
L_h (\chi_{out})_{h,\delta}^{w}(x',hD_{x'}) \phi_h^{H} = \\
T_h^{-1} \circ h^{\ast - (n-1)/2} \int    e^{i(y - y'')\cdot \eta/h}  e^{i\Phi(y'', y', v)/h} e^{-i\dist(y', x)/h}  \\ \times  \chi_+(\eta_1/h^\delta) a_1(y'', y', v, h) a_2( x,y' , h) \phi_h (x) \, dy' \, dv  \, dy'' \, dx \, d\eta + O(h^\infty). 
\end{gathered}\label{intexp}\end{equation}

The point of expressing $\phi_h$  in terms of itself  is that we then have an explicit representation of $\phi_h^H= \gamma_H \chi_{h^{-1}} \phi_h$ in terms of a oscillatory kernel with phase function $\dist(y', x)$, which has  oscillations of semiclassical frequency $\leq 1$. On the other hand, the $\chi_+$ term is supported where the semiclassical frequencies are at least $1 + h^{\delta}$. So the phase in \eqref{intexp} should be nonstationary on the support of the integral, allowing us to  perform integration by parts in the above integral. 

Let $\Psi$ denote the sum of the phase functions in \eqref{intexp}. Then, since $\Phi$ parametrizes the canonical relation 
$$
\big\{ (y'', \eta'', y', \eta') \mid (y'', \eta'') = \rho(y', \eta') \big\},
$$
where $\eta'$, respectively $\eta''$, denotes the dual variable to $y' \in H$, respectively $y'' \in \mathbb{R}^{n-1}$, 
 we have 
\begin{equation}\begin{gathered}
d_{v} \Psi = 0 \implies \rho(y', -d_{y'} \Phi(y'', y', v)) = (y'', d_{y''}\Phi(y'', y', v)) \\
d_{y''} \Psi = 0 \implies \eta = d_{y''}\Phi(y'', y', v) \\
d_{y'} \Psi = 0 \implies d_{y'} \Phi(y'', y', v) =  d_{y'} \dist(y', x).
\end{gathered}\end{equation}

Putting these together we find that
$$
d_{v, y'', y'} \Psi = 0 \implies \rho(y',-d_{y'} \dist(y', x)) = (y'', \eta).
$$
Since we chose $\rho$ such that $\rho^* \eta_1 = |\eta'|_{\tilde g}^2 - 1$, this shows that 
$$
d_{v, y'', y'} \Psi = 0 \implies \eta_1 = \Big| d_{y'} \dist(y', x) \Big|_{\tilde g}^2 - 1.
$$
It follows that we can write 
\begin{equation}
\eta_1 +1 -  \Big| d_{y'} \dist(y', x) \Big|_{\tilde g}^2 = \sum_i \Big( \gamma_i d_{v_i} \Psi + \gamma'_i d_{y'_i} \Psi + \gamma''_i d_{y''_i} \Psi \Big),
\label{eta1-identity}\end{equation}
where the $\gamma_i, \gamma'_i, \gamma''_i$ are smooth. 

Now consider $(x, y')$ such that $y' \in H$ and $(x, y', 0)$ is in the support of $a_2$. For such $(x, y')$, we denote by $e$ the unit vector in $T_{y'} M$ that generates the short geodesic between $y'$ and $x$, and 
write $\theta = \theta(x, y')$ for the angle in $T M$ (measured using the metric $g$) between the vector $e$ and the normal vector to $H$ at $y'$. By construction, $x$ and $y'$ cannot be closer than $(2C_0)^{-1} \epsilon_0$ together, so  $\theta$ is a smooth function of $x$ and $y'$. Then we can express
$$
\Big| d_{y'} \dist(y', x) \Big|_{\tilde g}^2 = \sin^2 \theta.
$$
Using \eqref{eta1-identity} we have 
\begin{equation}
\Big( \frac{h}{i} \frac{\gamma_i d_{v_i} + \gamma'_i d_{y'_i}  + \gamma''_i d_{y''_i} }{\eta_1 + \cos^2 \theta} \Big)^N e^{i\Psi(x, y, y', y'', \theta, \eta)/h} = e^{i\Psi(x, y, y', y'', \theta, \eta)/h}.
\label{ibp}\end{equation}
We insert this in \eqref{intexp} and  integrate by parts $N$ times. The derivatives are harmless (in the sense that they produce no negative powers of $h$) when they hit the factors $a_1 a_2 \chi_+(\eta_1/h^{\delta})$ or the $\gamma'_i, \gamma''_i, \gamma'''_i$  (note that there are no $\eta$ derivatives to fall on the $\chi_+$ factor). However,
 since $\theta$ is a function of $(x, y')$,  the $y'$ derivatives can hit the denominator of \eqref{ibp} and then we have to estimate more carefully. Consider a single $y'$ derivative hitting the  $\cos^2 \theta$ factor in the denominator of \eqref{ibp}. 
 In that case, we get an overall factor 
\begin{equation}
h \frac{2 \cos \theta \sin \theta \sum_i \gamma'_i d_{y'_i}\theta}{(\eta_1 + \cos^2 \theta)^2}.
\label{worstcase}\end{equation}
Notice that on the support of $\chi$, we have $\eta_1 \geq h^\delta/2$. The denominator can therefore be small when $\cos \theta$ is small. However, notice that we also have a factor of $\cos \theta$ in the numerator. We can estimate this term as follows: either $\cos \theta \leq h^{\delta/2}$ or $\cos \theta \geq h^{\delta/2}$. In the former case, we get that \eqref{worstcase} is bounded by 
$$
Ch \frac{h^{\delta/2}}{h^{2\delta}} = Ch^{1 - 3\delta/2}.
$$
In the latter case, we find that \eqref{worstcase} is bounded by 
$$
Ch \frac{\cos \theta}{\cos^4 \theta} \leq C h^{1 - 3\delta/2}.
$$
In either case, we gain a positive power of $h$ for each integration by parts provided that $\delta < 2/3$. Higher numbers of $y'$ derivatives hitting either the denominator in \eqref{ibp} or the $\cos \theta$ factors in \eqref{worstcase} can be estimated similarly. Thus, by integrating by parts sufficiently many times, we prove that for any $N$ we have 
$$
L_h(\chi_{out})_{h,\delta}^{w}(x',hD_{x'}) \phi_h^{H} = O(h^N), \text{ provided } \delta < 2/3.
$$

 In the case where the submanifold $H$ is totally geodesic, i.e.~ its second fundamental form vanishes, we can see by inspecting the proof that we can prove Proposition~\ref{prop:mass-decay} for all $\delta < 1$. To see this, consider the quantity 
$
d_{y'} \theta
$
from \eqref{worstcase}. Let $n = n_{y'}$ denote the unit normal vector at $y' \in H$, and let $e_x$ denote the unit length vector field pointing away from $x$, that is, in the direction of geodesics emanating from $x$. Thus $\cos \theta = n_{y'} \cdot e_x$. If we differentiate in $y'$ we find that 
$$
-(\sin \theta) d_{y'} \theta = \nabla_{y'} \big( n_{y'} \cdot e_x \big). 
$$
Suppose that $\cos \theta = 0$. Then $\sin \theta = 1$, so this factor can be ignored. 
More importantly, if $e_x$ is normal to $n_{y'}$, or equivalently tangent to $H$, then $H$ totally geodesic means that the whole geodesic generated by $e_x$ is contained in $H$. In particular, this implies that $x \in H$ and the vector field $e_x$ is tangent to $H$. Therefore, $n_{y'} \cdot e_x$ vanishes identically. So this derivative is zero. 

It follows that $d_{y'} \theta = 0$ when $\cos \theta = 0$. Since $\cos \theta$ vanishes simply, this implies that $d_{y'} \theta =  k(x, y')\cos \theta$ for some smooth function $k$. Thus in the case that $H$ is totally geodesic, we get an extra factor of $\cos \theta$ in the numerator of \eqref{worstcase}, leading to the conclusion that \eqref{worstcase} can be estimated by a constant times $h^{1-\delta}$. Hence, in this case, we can take any $\delta < 1$. 


\end{proof}

\begin{corollary}\label{cor:Ck}
The conclusion of Proposition~\ref{prop:mass-decay} can be strengthened to 
\begin{equation}
(\chi_{out})_{h,\delta}^{w} \phi_h^H = O(h^\infty) \text{ in } C^k(H)
\label{ext-mass-Ck}\end{equation}
for any $k \in \mathbb{N}$. 
\end{corollary}

\begin{proof} We apply $l$ derivatives to $L_h(\chi_{out})_{h,\delta}^{w}(x',hD_{x'}) \phi_h^{H}$. This brings down a factor of $h^{-l}$. Applying the argument above shows that the result is $O_{L^2}(h^\infty)$. The Sobolev embedding theorem then gives \eqref{ext-mass-Ck}, provided $l > k + (n-1)/2$. 
\end{proof}

\begin{remark}
It is worth noting that the upper limit of $2/3$ on the size of
$\delta$ is sharp. Consider the example of the unit disc in the plane,
e.g.~the set $\{ r < 1 \}$ in standard polar coordinates. Let $H$ be
the circle $\{ r = 1/2 \}$. Dirichlet eigenfunctions with eigenvalue
$\lambda^2$ take the form $f_n = c_n e^{in \theta} J_n(\lambda r)$, where $J_n$ is the standard Bessel function of order $n$,  and where $J_n(\lambda) = 0$. There are pairs $(n, \lambda)$ where 
$$
\lambda - 2n \in [-z_1 n^{1/3}, -z_2 n^{1/3}], \quad z_1, z_2 > 0. 
$$
That is, $2n$ is a bit bigger than $\lambda$, by an amount  $z \lambda^{1/3}$ where $z \in [z_1, z_2]$. This means that, at $H$, that is, when $r=1/2$,  we are near the turning point of $f_n$, where $f_n$ has Airy asymptotics given by \cite[9.3.43]{AS}. Normalizing the eigenfunction in $L^2$ requires that $c_n \sim n^{1/6}$. Applying a semiclassical derivative means we gain a factor $n^{-1/3}$ since this is the length scale on which solutions of Bessel's equation oscillate near the turning point. 
 This implies that 
$$
f_n'(n - z n^{1/3}) \sim n^{-1/6} \mathrm{Ai}( 2^{1/3}z),
$$
where $\mathrm{Ai}$ is the Airy function. In particular, we see that on $H$, the restriction of an eigenfunction can contain semiclassical frequencies of the size $1 + z h^{2/3}$, $z > 0$, which decay only polynomially as $h \to 0$. However, the rapid decay of solutions of Airy's equation as $z \to \infty$ means that frequencies of the size $1 + c h^{2/3 - \epsilon}$, $c > 0$, $\epsilon > 0$, decay rapidly  as $h \to 0$, in agreement with Proposition~\ref{prop:mass-decay}. 
\end{remark}

\begin{remark} It is interesting to contrast the sharpness of the exponent $\delta = 2/3$  in mass concentration on $H$ with the situation on $M$. In the case of $M$, (see (\ref{massM})) a parametrix computation   using the 2-microlocal calculus associated to $S^* M$ shows that the exterior mass of eigenfunctions in the region where $|\xi|_g \geq 1 + C h^{\delta}$ is $O(h^\infty)$ for {\em any} $\delta < 1$. When $H$ is totally geodesic, the same is true for the exterior mass of the restricted eigenfunctions $\phi_h |_{H}.$ However,  when $H$ has positive definite second fundamental form, the disc example above shows that one must restrict to the range $0 \leq \delta < 2/3.$ 
\end{remark}

\section{Improved Neumann estimate} \label{mainthm}
In this section, we return to the computation started in the introduction. We have 
\begin{align} \label{Rellich}
\frac{i}{h}  \int_{M_-} [-h^2 \Delta -1, &\chi(x_n) h D_n] \phi_h \overline{\phi_h}
dx \nonumber \\
 = \int_H ((hD_n)^2  \phi_h)|_H \overline{\phi_h}|_H d \sigma_H 
& + \int_H (h D_n \phi_h)|_H \overline{hD_n \phi_h}|_H d \sigma_H .
\end{align}
and, as outlined in the Introduction, this yields 
\begin{equation}\begin{aligned}
 \int_H (1 + h^2 \Delta_H) \phi_h^H \overline{\phi_h^H} d \sigma_H  \quad + \int_H | \phi_h^{H,\nu} |^2 d \sigma_H = \O(1).
\end{aligned}\label{O1est}\end{equation}
In order to bound the Neumann data from above, we therefore need to show the
first term on the left hand side is essentially positive.

For this we now use our small scale decomposition.  Let $\chi_{in},
\chi_{tan}, \chi_{out}$ be as before, and
\[
1 = (\chi_{in})_{h,\delta}^w + (\chi_{tan})_{h,\delta}^w +(\chi_{out})_{h,\delta}^w 
\]
be the corresponding 2-microlocal partition of unity, with $\delta$ chosen in the range $(1/2, 2/3)$.  We have
\begin{align*}
 \int_H (1 + h^2 \Delta_H) \phi_h^H \overline{\phi_h^H} &d \sigma_H \\
 = \int_H (1 + h^2 \Delta_H) &(\chi_{in})_{h,\delta}^w\phi_h^H
\overline{\phi_h^H} d \sigma_H  \quad + \int_H (1 + h^2 \Delta_H) (\chi_{tan})_{h,\delta}^w
\phi_h^H \overline{\phi_h^H} d \sigma_H \\
 \quad + 
\int_H &(1 + h^2 \Delta_H) (\chi_{out})_{h,\delta}^w  \phi_h^H \overline{\phi_h^H} d \sigma_H
\\
 = \int_H (1 + h^2 \Delta_H) &(\chi_{in})_{h,\delta}^w\phi_h^H
\overline{\phi_h^H} d \sigma_H  \quad + \int_H (1 + h^2 \Delta_H) (\chi_{tan})_{h,\delta}^w
\phi_h^H \overline{\phi_h^H} d \sigma_H \\
 & + \O(h^\infty)
\end{align*}
where we used Proposition~\ref{prop:mass-decay} (or really Corollary~\ref{cor:Ck}) in the last line. 

As estimated previously, on the support of $\chi_{in}$, we have $1 -
R(x', 0, \xi') \geq h^\delta$.  Without loss in generality, assume
$\chi_{in} = \psi^2$ for some $\psi \geq 0,$ 
\[
\psi = \psi( (R(x',0,\xi')-1)/h^\delta) \in S^0_{\Sigma_H, \delta}.
\]
Let $\tchi \in S^0_{\Sigma_H, \delta}$ satisfy $\tchi \equiv 1$ on
$\supp \chi_{in}$ with slightly larger support, say on a set where $1-R(x',
0, \xi') \geq h^\delta/M$ for some large $M$.  Observe that then  
\[
\ell = (1 - \tchi) + h^{-\delta} \tchi \cdot (1-R(x',0,\xi') )  \in
S^\delta_{\Sigma_H,\delta}
\]
satisfies $\ell \geq c_0 >0$.  If $L = \ell^w$, then the G{\aa}rding
inequality (Lemma \ref{L:garding}) implies
\[
\lll L u, u \rrr \geq (c_0 - C h^{1-\delta}) \| u \|^2.
\]
Further,
\begin{align*}
L \chi_{in}^w \phi_h^H & = L (\psi^w)^* \psi^w \phi_h^H \\
& = (\psi^w)^* L  \psi^w \phi_h^H + [L,(\psi^w)^*]   \psi^w \phi_h^H.
\end{align*}
Since, on the support of $\psi$, $\ell = h^{-\delta} (1-R)$, $L$
commutes to leading order with $\psi$, so a crude estimate on the
commutator gives
\[
\|  [L,(\psi^w)^*]   \psi^w \phi_h^H \| = \O(h^{-\delta} h^{3(1-\delta)}) \|
\phi_h^H \|.
\]
The powers of $h$ in this estimate come from $h^{-\delta}$ in the
definition of $\ell$, and 3 powers of $h^{1-\delta}$ because
derivatives can lose $h^{-\delta}$, and in the Weyl calculus, the
second order term in the commutator vanishes by anti-symmetry.  
Hence
\begin{align*}
\lll (1+h^2\Delta_H) (\chi_{in})^w \phi_h^H, \phi_h^H \rrr & =
h^\delta \lll L
(\chi_{in})^w \phi_h^H, \phi_h^H \rrr + \O(h^\infty) \| \phi_h^H \|^2\\
& = h^\delta \lll  (\psi^w)^* L  \psi^w \phi_h^H, \phi_h^H \rrr +
\O(h^{3-3\delta}) \| \phi_h^H \|^2 \\
& = h^\delta \lll  L  \psi^w \phi_h^H, \psi^w \phi_h^H \rrr  +
\O(h^{3-3\delta}) \| \phi_h^H \|^2 \\
& \geq h^\delta  (c_0 - C h^{1-\delta} ) \| \psi^w \phi_h^H \|^2
- C h^{3-3\delta}  \| \phi_h^H \|^2.
\end{align*}


On the other hand, on the support of $\chi_{tan}$, we have $| 1 -
R(x', 0, \xi') | \leq C_2 h^\delta$, so that
\[
\left|  \int_H (1 + h^2 \Delta_H) (\chi_{tan})_{h,\delta}^w
\phi_h^H \overline{\phi_h^H} d \sigma_H  \right| \leq C_2 h^\delta
 \|
\phi_h^H \|^2.
\]
Combining these two estimates, we have
\begin{align*}
& \int_H (1 + h^2 \Delta_H) \phi_h^H \overline{\phi_h^H} d \sigma_H \\
& \geq C_1 h^{\delta} \int_H (\chi_{in})_{h,\delta}^w\phi_h^H
\overline{\phi_h^H} d \sigma_H -
C_1' h^{3-3\delta}  \| \phi_h^H \|^2 \\
& \quad 
-  C_2 h^\delta
\|
\phi_h^H\|^2
+ \O(h^\infty) \\
& \geq -C h^\delta \int_H | \phi_h^H |^2 d \sigma_H,
\end{align*}
since the exterior term is $\O(h^\infty)$, and
provided $3-3\delta \geq \delta$, or $\delta \leq 3/4$ (recall we have
already assumed $\delta <2/3$).  
 Employing  the $h^{-1/4}$ bound of
Burq-G\'erard-Tzvetkov \cite{BGT}, we get
\begin{align*}
\int_H (1 + h^2 \Delta_H) \phi_h^H \overline{\phi_h^H} d \sigma_H
\geq -C h^{\delta-1/2}.
\end{align*}

Since we chose $\delta > 1/2$, this gives, in combination with \eqref{O1est}, 
\begin{align*}
& -C h^{\delta - 1/2} + \int_H | \phi_h^{H,\nu} |^2 d \sigma_H \\
& \leq \int_H (1 + h^2 \Delta_H) \phi_h^H \overline{\phi_h^H} d \sigma_H + \int_H | \phi_h^{H,\nu} |^2 d \sigma_H \\
& = \O(1),
\end{align*}
or, rearranging, 
\[
\int_H | \phi_h^{H,\nu} |^2 d \sigma_H = \O(1),
\]
which proves Theorem~\ref{main}. 

\begin{remark} \label{bterm}
Notice that this computation also shows that 
$$
 \int_H (1 + h^2 \Delta_H) \phi_h^H \overline{\phi_h^H} d \sigma_H = O(1)
$$
since all the other terms in \eqref{O1est} are $O(1)$. 
\end{remark}


\section{Schr\"odinger eigenfunctions: proof of Theorem \ref{main2}} \label{schroedinger}
 Given the semiclassical Schr\"odinger operator $P(h) = -h^2 \Delta_g + V(x),$ we denote the principal symbol by $p(x,\xi) = |\xi|_g^2 + V(x).$ In the following, we work in a collar neighbourhood around $H$ and continue to denote the corresponding Fermi coordinates by $(x',x_n)$ with $H = \{ x_n=0 \}.$ Given a regular energy value $E \in \R,$ we let $\Sigma(E) = \{(x,\xi) \in T^*M; p(x,\xi) = E \}$ be the corresponding level set and
 $$\Sigma_{H}(E) = \{ (x',\xi') \in T^*H; p(x',0;\xi',0) = E \}$$
 be the restriction to $T^*H.$
  Since $H \subset \{ x \in M; V(x) < E \},$ it follows that
 $d_{\xi'}p(x',\xi') \neq 0$ for all $(x',\xi') \in \Sigma_H(E)$ and so the level set $\Sigma_H(E) \subset T^*H$ is a smooth hypersurface.  Consider $L^2$-normalized eigenfunctions $\phi_h \in C^{\infty}(M)$ with
 \begin{equation} \label{geneigenfn}
 P(h) \phi_h = E(h) \phi_h, \,\, |E(h) - E| = O(h).\end{equation}
 The relevant semiclassical 2-microlocal cutoffs to $\Sigma_H(E)$ are 
 $$(\chi_{out,in,tan})^{w}_{h,\delta} = Op_{h}^w    \, \chi_{out,in,tan} \Big( \frac{| \xi' |_{g}^2 + V(x',0) - E}{h^{\delta} }\Big) , \,\, 0 \leq \delta <1.$$
 Since $\Sigma_H(E) \subset T^*H$ is a hypersurface of real principal
 type, there is a natural symbol calculus for these operators just as in subsection \ref{2microlocal} with the associated sharp G{\aa}rding and $L^2$-boundedness results. As  in Lemma \ref{normal form}, the key here is a local normal form result which says that there are compactly-supported $h$-pseudodifferential operators $L_h \in Op_{h}^w(S^{0,-\infty}(T^*H))$  and intertwining, compactly-supported  $h$-Fourier integral operator $T_h: C^{\infty}_{0}(T^*H) \to C^{\infty}_{0}(\R^{n-1})$ with the property that with $\chi = \chi_{tan,in,out},$
 \begin{equation} \label{normal form 2}
 T_h  L_h \chi^{w}_{h,\delta}  T_h^{-1} = \chi^{w}_{h} \left(\frac{\eta_1}{h^{\delta}} \right). \end{equation}
 The proof of Theorem \ref{main2} follows as in the homogeneous case with one  minor change regarding the exterior restricted eigenfunction mass estimates on $h^{\delta}$-scales which we now explain.  
 \subsection{Exterior mass estimates}
The main change we need to make is to replace Sogge's approximate projection operator $\chi_\lambda$ by the corresponding operator depending on the potential $V$. To this end, we define the action $A(x, y)$, for $x, y \in M$ sufficiently close, to be the integral of the quantity $L(x, \xi) := |\xi|^2_g - V(x)$ along a bicharacteristic of $P$ starting
at $(y, \eta)$ and ending at $(x, \xi)$, contained in the energy surface $\{ p = E \}$. (Note that for $x \neq y$ sufficiently close, there is a unique $\eta$ such that there is a short bicharacteristic starting at $(y, \eta)$ and reaching $T^*_xM$ --- see the proof of the lemma below.) Due to our condition that $V < E$ on $H$, $A$ is comparable to the distance function when $x, y$ are sufficiently close to each other and to $H$. 
We then have: 

\begin{lemma}\label{action} Let $\chi$ be a Schwartz function with $\chi(0) = 1$ and with Fourier transform $\hat \chi$ having support in the interval $[\epsilon_0/2, \epsilon_0]$ for sufficiently small $\epsilon_0$. Then the operator 
$$
\chi_{V, \lambda} := \chi\big( \frac{ P(h) - E}{2h} \big)
$$
has kernel of the form 
\begin{equation}
\chi_{V, \lambda}(x,y) = \lambda^{(n-1)/2} a_3(y, x, \lambda) e^{-i\lambda A(x, y)} + R(x, y, \lambda), \quad \lambda = h^{-1}, 
\label{chivl}\end{equation}
where $a_3$ is supported where $A(x, y) \in [(2 C_0)^{-1} \epsilon_0, 2C_0 \epsilon_0]$ for some $C_0 > 1$, and $R$ is smooth with all derivatives $O(\lambda^{-N})$ for every $N$. 
\end{lemma}

\begin{proof}
We express $\chi_{V, \lambda}$ in terms of the semiclassical propagator:
\begin{equation}
\chi_{V, \lambda} = (2\pi)^{-1} \int_0^\infty e^{it(P(h) - E)/2h} \hat \chi(t) \, dt.
\label{t-int}\end{equation}
The operator $H(t) e^{itP(h)/2h}$, where $H(t)$ is the Heaviside function, is the forward fundamental solution for the operator $2hD_t - P(h)$, and its microlocal structure is well understood. It is associated to two Lagrangian submanifolds $\Lambda_0$ and $\Lambda$, where $\Lambda_0$ is the conormal bundle to $\{ t = 0, x = y \} \subset T^* (M \times \R)$ and $\Lambda$ is the flowout, in the direction of positive time, from the intersection of $\Lambda_0$ and the characteristic variety $\Sigma_V$ of $2h D_t - P(h)$. This flowout is, by definition, the union of bicharacteristics starting at $\Sigma_V \cap \Lambda_0$, and satisfying the ODE
\begin{equation}\begin{aligned}
\dot t &= 2 \qquad 
&\dot \tau &= 0 \\
\dot x &= -\frac{\partial p}{\partial \xi} \quad 
&\dot \xi &= \frac{\partial p}{\partial x} \\
\dot y &= 0 \quad 
&\dot \eta &= 0.
\end{aligned}\end{equation}
It is not hard to see from this that $(x, y, t)$ form coordinates on $\Lambda$ for $t > 0$ small. Indeed, it is clear that $(y, \eta, t)$ form coordinates, since $(y, \eta)$ form coordinates on $\Lambda \cap \Lambda_0$ and $t$ can be used as a parameter along bicharacteristics. On the other hand, the equations above show that
\begin{equation}
x_i(t) = t \sum_j g^{ij}(x)\eta_j + O(t^2)
\label{xit}\end{equation}
showing that $\partial x / \partial \eta$ is nonsingular for $t > 0$ small, so we may take $(x, y, t)$ instead of $(y, \eta, t)$. So there is a unique function $\Phi(x, y, t)$, smooth for $t > 0$ parametrizing $\Lambda$ locally near $x=y$.  
Moreover, the value of  $\Phi$  is given by Hamilton-Jacobi theory,  by solving the ODE along bicharacteristics
\begin{equation}
\dot \Phi = -\xi \cdot \frac{\partial p}{\partial \xi} - 2\tau + p  = - \Big( g^{ij}(x) \xi_i \xi_j - V(x) \Big) - 2\tau = -L(x, \xi) - 2\tau
\label{Phi-value}\end{equation}
with initial value $\Phi = 0$ at $\Lambda_0 \cap \Sigma_V$. 

We now put this expression into \eqref{t-int} to obtain the kernel of $\chi_{V, \lambda}$:
\begin{equation}
\chi_{V, \lambda} = (2\pi)^{-1} (2\pi h)^{-n/2} \int_0^\infty e^{i\Phi(x, y, t)/h} b(x, y, t, h) e^{-it/2h} \hat \chi(t) \, dt.
\label{t-int-2}\end{equation}
We claim that $\partial^2_{tt} \Phi \neq 0$ for $t > 0$ small. To see this write $\tau(x, y, t)$ and $\eta(x, y, t)$ for the value of $\tau$, respectively $\eta$, on $\Lambda$ at the point parametrized by $(x, y, t)$. 
Then $d_t \Phi = \tau(x, y, t)$, so we need to show that $d_t \tau \neq 0$. We can rotate coordinates so that $g^{ij}(y)$ is diagonal at $y_0$ and $\eta(x_0, y_0, t_0)$ is a multiple of $(1, 0, \dots, 0)$. Write $x$ as a function $x = X(y, \eta_2, \dots, \eta_n, \tau, t)$ since $\eta_1$ is determined by $\eta_2, \dots, \eta_n, \tau$ on $\Lambda \subset \Sigma_V$, and since 
$x$ is determined by following the bicharacteristic starting at $(y, \eta)$ for time $t$. Then we have 
$$
0 = \frac{dx}{dt} = \sum_{j=2}^n \frac{\partial X}{\partial \eta_j} \frac{\partial \eta_j}{\partial t} + \frac{\partial X}{\partial \tau} \frac{\partial \tau}{\partial t} +  \frac{\partial X}{\partial t}.
$$
We have $|\partial_t X| = \sqrt{E - V}$, which is bounded away from zero using our assumption on $V$. On the other hand, $\partial_{\eta_j} X = O(t + |y-y_0|)$ for $j \geq 2$ and for $(x, y, t)$ near $(x_0, y_0, 0)$ using \eqref{xit} and the assumptions on $g^{ij}(y_0)$ and $\eta(x_0, y_0, t_0)$. It follows from the above identity that $\partial_t \tau \neq 0$ for small $t$, showing that $\partial^2_{tt} \Phi \neq 0$, as claimed. 
So we can perform stationary phase in the $t$ variable in \eqref{t-int-2}, obtaining the phase function 
$$
\Phi(x, y, t(x,y)) - \frac{Et(x, y)}{2}
$$
where $t(x,y)$ is the stationary point. Note that stationarity in $t$ requires that $\tau = E/2$, which implies that $p=E$ since $\Lambda$ is contained in $\Sigma_V$. Inserting this in \eqref{Phi-value} and using the constancy of $\tau$ along bicharacteristics gives the value
$$\Phi(x, y, t(x,y)) = -A(x,y) + \frac{Et(x, y)}{2}
$$
according to the definition of the action $A$ above. 
This shows that the phase function is as claimed in the lemma. The power of $\lambda$ follows from \eqref{t-int-2} and stationary phase in $t$, and the properties of $a_3$ and of $R$ follow as in the homogeneous case. 
\end{proof}

\begin{remark} The form of the parametrix for $e^{it(P(h) - E)/2h}$ in \eqref{t-int-2} also follows from the more
standard integral representation with phase $S(t,x,\eta) - y \cdot \eta$ (see \cite[Section 10.2]{Zw-book}) by stationary phase in $\eta$   for $t \in [(2C_0)^{-1}\epsilon_0, 2C_0 \epsilon_0]$ with $\epsilon_0 >0$ sufficiently small. Here, $S(t,x,\eta)$ solves the Hamilton-Jacobi equation $\partial_t S =  p(x,d_xS) $ with   $S(0,x,\eta) = x \cdot \eta.$
\end{remark}

\begin{remark} The reason for choosing the factor  $2$ in the denominator of the expression $\chi \big( (P(h) - E)/2h \big)$ is to match as closely as possible to Sogge's approximate projection operator $\chi_\lambda = \chi(\sqrt{\Delta} - \lambda)$. This could be written 
$\chi \big( (\Delta - \lambda^2)/(\sqrt{\Delta} + \lambda) \big)$, which is very close to 
$\chi \big( (\Delta - \lambda^2)/2\lambda \big)$ when $\sqrt{\Delta}$ is localized close to $\lambda$. In semiclassical notation this is $\chi \big( (h^2 \Delta - 1)/2h \big)$. 
\end{remark}

\begin{remark} The same representation \eqref{chivl} holds for the operator $\chi \big( (P(h) - E(h))/2h \big)$, where $E(h)$ lies in the interval $[E - Ch, E + Ch]$. This is clear since replacing $E$ with $E + ch$ amounts to a translation of $\chi$, or equivalently a modulation of $\hat \chi$, which does not change any essential properties of $\chi$. Moreover, all statements in Lemma~\ref{action} hold uniformly for $c \in [-C, C]$.
\end{remark}

 Now we follow the argument of Section~\ref{extmass} almost verbatim. 
Using the identity
$\chi\big( \frac{ P(h) - E(h)}{h} \big) \phi_h = \phi_h$ and with the same notation as in section \ref{extmass}, one has the local formula
\begin{align}\label{outsidemass2}
L_h (\chi_{out})_{h,\delta}^{w}(x',hD_{x'}) \phi_h^{H} &=h^*T_h^{-1}  \circ  \int    e^{i(y - y'')\cdot \eta/h}  e^{i\Phi(y'', y', v)/h} e^{-i A(y', x)/h}  \end{align}
 $$ \times  \chi_+(\eta_1/h^\delta) a_1(y'', y', v, h)  \hat{\chi}(t) a_3(y', x, h) \phi_h(x) \, dy' \, dv  \, dy'' \, dx \, d\eta + O(h^\infty). $$
 \medskip

Let $\Psi$ denote the sum of the phase functions in \eqref{outsidemass2}. 
 We integrate by parts  in $(y',y'',v)$ and we compute the critical set in these variables:
\begin{equation}\begin{gathered}
d_{v} \Psi = 0 \implies \rho(y', -d_{y'} \Phi(y'', y', v)) = (y'', d_{y''}\Phi(y'', y', v)) \\
d_{y''} \Psi = 0 \implies \eta = d_{y''}\Phi(y'', y', v) \\
d_{y'} \Psi = 0 \implies d_{y'} \Phi(y'', y', v) =  d_{y'} A(y', x).
\end{gathered}\end{equation} 
This implies that $\rho(y',-d_{y'}A) = (y'',\eta)$ with $\rho^*\eta_1 = |\xi'|_g^2 + V(y',0) - E$ and so,
$$ d_{v,y',y''} \Psi = 0 \implies \eta_1 = |d_{y'}A|^2_g + V(y',0) - E.$$

Taylor expansion yields smooth functions $\gamma_i,\gamma_i', \gamma_i''; 1 \leq i \leq n-1$  with
\begin{align} \label{crit}
 &e^{i\Psi(x, y, y', y'',t, r\omega, \eta)/h}  \nonumber \\
 &= \Big( \frac{h}{i} \Big)^N \cdot \Big( \frac{  \gamma_i d_{v_i} + \gamma'_i d_{y'_i}  + \gamma''_i d_{y''_i}  }{  \eta_1 - (  |d_{y'}A|^2_{g} + V(y',0) -E ) } \Big)^N e^{i\Psi(x, y, y', y'', v, \eta)/h} \nonumber \\
 &= \Big( \frac{h}{i} \Big)^N \cdot \Big( \frac{  \gamma_i d_{v_i} + \gamma'_i d_{y'_i}  + \gamma''_i d_{y''_i}  }{  \eta_1 +  |d_{y_n}A|^2 } \Big)^N e^{i\Psi(x, y, y', y'', v, \eta)/h} 
\end{align}
since $A$ satisfies the Hamilton-Jacobi equation $ |d_{y_n}A|^2 + |d_{y'} A|_g^2 = E - V$. 

As in the homogeneous case, using (\ref{crit}), we integrate by parts.  Differentiation in the $v,y''$ coordinates is harmless in that it does not produce any singular behaviour in $h.$ Differentiation in the $y'$-variables is more subtle  in the case where the $d_{y'}$ derivatives hit the term in the denominator in (\ref{crit}) involving $d_{y_n}A,$ one must bound a ratio of the form
\begin{equation} \label{critfraction}
 \frac{ h  \, |d_{y_n} A|  \cdot   |d_{y'}d_{y_n}A| }{ \Big( \eta_1 +  |d_{y_n}A|^2  \Big)^2}.\end{equation}
Here, $|d_{y_n}A|^2$ plays the role of $\cos ^2 \theta$ in the homogeneous case. We split (\ref{critfraction}) into the two cases $|d_{y_n}A| \geq h^{\delta/2}$ and $|d_{y_n}A| \leq h^{\delta/2}.$  In both cases, we get an $O(h^{1-3\delta/2})$ bound for (\ref{critfraction}) and this imposes the constraint that $\delta < 2/3$ as in Proposition \ref{prop:mass-decay}. To summarize, as in the homogeneous case,
\begin{equation}\label{massH2}
(\chi_{out})_{h,\delta}^{w} \phi_h^H = O(h^\infty) \text{ in } L^2(H).
\end{equation}

 As for the Rellich formula analogue of (\ref{Rellich}), we note that the real-valued potential $V$ cancels in Green's formula to give precisely the same RHS  as in (\ref{Rellich}).  Indeed, since $V \in C^{\infty}(M;\R)$ and $(-h^2 \Delta + V - E(h)) \phi_h = 0,$
 \begin{align} \label{Rellich2}
\frac{i}{h}  \int_{M_-} [-h^2 \Delta + V - E(h),  \, \chi(x_n)  hD_n] \phi_h \overline{\phi_h}
dx \nonumber \\
 = \int_H ((hD_n)^2  \phi_h)|_H \overline{\phi_h}|_H d \sigma_H 
& + \int_H (h D_n \phi_h)|_H \overline{hD_n \phi_h}|_H d \sigma_H .
\end{align}
 
 Given the mass estimate in (\ref{massH2}), the G{\aa}rding and $L^2$-boundedness results in the 2-microlocal operator calculus and the commutator formula in (\ref{Rellich2}),  the rest of the proof of Theorem \ref{main2} follows in the same way as in Theorem \ref{main}. \qed

\section{Optimality: an example with spherical harmonics}
\label{S:sharp}
\def\p{\partial}
In this section, we show that the Neumann data restriction estimate in
Theorem \ref{main} is optimal in the case of a highest weight spherical
harmonic.  To fix our notation, we consider $M = \mathbb{S}^2$ with
the parametrization in $\reals^3$:
\[
(x_1, x_2, x_3) = ( \sin \phi \cos \theta, \sin \phi \sin \theta, \cos
\phi),
\]
where $0 \leq \phi \leq \pi$ is the angle from the north pole and $0
\leq \theta \leq 2 \pi$ is the angle in the $x_1x_2$ plane measured
from the $x_1$ axis.  The induced metric is the usual spherical
metric:
\[
g = d \phi^2 + \sin^2 \phi d \theta^2,
\]
and the volume form is
\[
dV = \sin \phi d \phi d \theta.
\]
The Laplacian is the usual angular part of the polar Laplacian:
\[
-\Delta_g = -\frac{1}{\sin \phi} \p_\phi \sin \phi \p_\phi -
\frac{1}{\sin^2 \phi } \p_\theta^2.
\]
The eigenfunctions for $-\Delta_g$ are homogeneous harmonic
polynomials in $\reals^3$ restricted to $M$.  The highest weight harmonic of
order $k$ is
given by restricting the polynomial $(x_2 + i x_1)^k$ to $M$.  Let
\[
u_k = (x_2 + i x_1)^k |_M = i^k \sin^k \phi e^{-ik \theta}.
\
\]
As an eigenfunction, it satisfies
\[
-\Delta_g u_k = k(k+1) u_k.
\]
As written, $u_k$ is of course not normalized.  Let us compute
\begin{align*}
\| u_k \|^2_{L^2(M)} & = \int_0^{2 \pi } \int_0^\pi | u_k |^2 \sin
\phi d \phi d \theta \\
& = \int_0^{2 \pi } \int_0^\pi \sin^{2k+1}
\phi d \phi d \theta \\
& = 2 \pi  \int_0^\pi \sin^{2k+1}
\phi d \phi ,
\end{align*}
which, after a computation, is seen to be equal to
\[
4 \pi \left( \frac{2k}{2k+1} \right) \left( \frac{2k-2}{2k-1} \right)
\left( \frac{2k-4}{2k-3} \right) \cdots \left( \frac{2}{3} \right).
\]

There is an orthonormal basis of the tangent space at any point of $M$
given by the vectors
\[
X = \frac{1}{\sin \phi} \p_\theta, \text{ and } Y = \p_\phi,
\]
since $g(X,X) = g(Y,Y) = 1$ and $g(X,Y) = 0$.  If we let $H \subset M$
be the periodic geodesic in the $x_1x_3$ plane originating from the north pole, then $H$ is parametrized by
\[
(x_1,x_2,x_3) =  ( \sin \phi , 0, \cos
\phi),
\]
for $0 \leq \phi \leq 2\pi$.  Notice there is no ambiguity with our
chart by taking $\phi$ in this extended range since we have frozen
$\theta = 0$.  Clearly the vector $Y$ is tangent to $H$,
so the normal derivative of $u_k$ will be given by $X u_k$.  We want
to compute $\| X u_k \|_{L^2(H)}$.  The induced metric on $H$ is the
usual circle metric $d \phi^2$, so the volume form is just $d \phi$.
We compute:
\[
X u_k = i^{k-1} k \sin^{k-1} \phi e^{-i k \theta}.
\]
Hence we want to compute
\begin{align*}
\| X u_k \|_{L^2(H)}^2 & = \int_0^{2\pi} k^2 \sin^{2k-2} \phi d \phi.
\end{align*}
Another long computation shows this to be equal to
\[
k^2 \pi \left( \frac{2k-3}{2k-2} \right) \left( \frac{2k-5}{2k-4} \right)
\left( \frac{2k-7}{2k-6} \right) \cdots \left( \frac{3}{4} \right).
\]

In order to show Theorem \ref{main} is sharp, we are interested in
bounding the ratio
\[
\frac{ \| X u_k \|_{L^2(H)}^2}{k^2 \| u_k \|^2_{L^2(M)}}
\]
from below.  We compute
\begin{align*}
\frac{ \| X u_k \|_{L^2(H)}^2}{k^2 \| u_k \|^2_{L^2(M)}} & = \frac{\pi \left( \frac{2k-3}{2k-2} \right) \left( \frac{2k-5}{2k-4} \right)
\left( \frac{2k-7}{2k-6} \right) \cdots \left( \frac{3}{4} \right)}{4 \pi \left( \frac{2k}{2k+1} \right) \left( \frac{2k-2}{2k-1} \right)
\left( \frac{2k-4}{2k-3} \right) \cdots \left( \frac{2}{3} \right)} \\
& = \frac{1}{4 \left( \frac{2k}{2k+1} \right) \left( \frac{2k-2}{2k-1}
  \right) } \prod_{j=1}^{k-2} \frac{  \left( \frac{2k-2j-1}{2k-2j}
  \right) }{ \left( \frac{2k-2j-2}{2k-2j-1}
  \right) } \\
& = \frac{1}{4 \left( \frac{2k}{2k+1} \right) \left( \frac{2k-2}{2k-1}
  \right) } \prod_{j=1}^{k-2} \frac{  \left( {2k-2j-1}
  \right)^2 }{ \left( {2k-2j-2}  \right) \left( 2k-2j \right)}.
\end{align*}
We observe that $(2k -2j -1)^2 = (2k-2j-2)(2k-2j) + 1$, so each factor
in the product is bounded below by $1$.  Hence the whole product is
bounded below by a positive constant, independent of $k$.  This shows
Theorem \ref{main} is sharp.

\begin{remark} In fact, Theorem~\ref{main} is sharp for any Riemannian manifold $(M,g)$ and any hypersurface $H$. To see this, we note that standard wave equation methods give an asymptotic of the form
$$
\sum_j \rho(\lambda - \lambda_j) \big| \phi_h^{H, \nu}(x) \big|^2 \sim \frac{\lambda^{n-1}}{n \operatorname{vol}(M)},
$$
where $\rho$ is a function with smooth compactly supported Fourier transform, with $\hat\rho(t) = 1$ for $t$ in  a neighbourhood of $0$. See for example \cite{Ho}. 
This implies that for a sufficiently large $C$, 
$$
\sum_{\lambda_j \in [\lambda, \lambda + C]}  \big| \phi_h^{H, \nu}(x) \big|^2 \geq c\lambda^{n-1}. 
$$
Integrating over $H$ we find that 
$$
\sum_{\lambda_j \in [\lambda, \lambda + C]}  \big\| \phi_h^{H, \nu} \big\|_{L^2(H)}^2 \geq c\lambda^{n-1}. 
$$
Since the number of $\lambda_j$ in the range $[\lambda, \lambda + C]$ is at most $C' \lambda^{n-1}$, we can choose for each $m$ an eigenfunction $\phi_{\lambda_j}$ such that 
$\lambda_j \in [m, m+C]$ and $\big\| \phi_h^{H, \nu} \big\|_{L^2(H)}^2 \geq c' > 0$, showing the optimality of Theorem~\ref{main} for $(M, g)$ and $H$.
\end{remark}

\end{document}